\documentclass[11pt,letterpaper,reqno]{amsart}
\usepackage{tikz}
\usetikzlibrary{positioning, shapes.geometric, arrows.meta}
\usepackage{amssymb}
\usepackage{amsmath}
\usepackage{amsthm}
\usepackage{amsfonts}
\usepackage{bbm}
\usepackage{graphicx}
\usepackage[T1]{fontenc}
\usepackage{doi}
\usepackage{bookmark}
\usepackage{hyperref}
\hypersetup{pdfstartview={FitH}}

\newtheorem{thm}{Theorem}[section]
\newtheorem{cor}[thm]{Corollary}
\newtheorem{lem}[thm]{Lemma}

\newtheorem{prop}[thm]{Proposition}

\newtheorem{rem}[thm]{Remark}

\newcommand{\tr}{\text{tr}}
\newcommand{\diag}{\text{diag}}

\makeatother

\begin{document}

\title[Schoenberg type inequalities]{Schoenberg type inequalities}

\author[Quanyu Tang]{Quanyu Tang}

\address{School of Mathematics and Statistics, Xi'an Jiaotong University, Xi'an 710049, P. R. China}
\email{tang\_quanyu@163.com}



\begin{abstract}
In the geometry of polynomials, Schoenberg's conjecture, now a theorem, is a quadratic inequality between the zeros and critical points of a polynomial whose zeros have their centroid at the origin. We call its generalizations to other orders Schoenberg type inequalities. While inequalities of order four have been previously established, little is known about other orders. In this paper, we present a Schoenberg type inequality of order six, as well as a novel inequality of order one, representing the first known result in the odd-order case. These results partially answer two open problems posed by Kushel and Tyaglov. We also make a connection to Sendov's conjecture. 
\end{abstract}

\maketitle

\section{Introduction} 
The relative location and other relations between the zeros and critical points of a polynomial have been a topic of recurring interest in the geometry of polynomials~\cite{Mar66, QS02, Tao22}. Let  $$p(z)=z^n+a_1z^{n-1}+\cdots+a_{n-1}z+a_n  $$ be a complex polynomial of degree $n\ge 2$. Assume that the zeros of \( p(z) \) are \( z_1, z_2, \ldots, z_n \), and the \emph{critical points} of \( p(z) \), i.e., the zeros of \( p'(z) \), are \( w_1, w_2, \ldots, w_{n-1} \).

In 1986, Schoenberg~\cite{Sch86} conjectured that if the centroid of the zeros of a polynomial lies at the origin, i.e., $\sum_{j=1}^{n} z_{j} = 0$, then
\begin{eqnarray}\label{s0}
\sum_{k=1}^{n-1} \left|w_{k}\right|^2 \le \frac{n-2}{n} \sum_{j=1}^{n} |z_{j}|^2,
\end{eqnarray}
with equality if and only if all the zeros of \( p(z) \) lie on a straight line in the complex plane.
Subsequently, it was shown in~\cite{BIS99, Kat99} that Schoenberg's conjecture is equivalent to the following general form:
\begin{equation}\label{s}
\sum_{k=1}^{n-1}\left|w_{k}\right|^2 \le \frac{n-2}{n}\sum_{j=1}^{n}|z_{j}|^2 + \frac{1}{n^2}\left|\sum_{j=1}^{n}z_{j}\right|^2.
\end{equation}
Inequality~\eqref{s} was first confirmed by Pereira \cite{Per03} and independently by Malamud \cite{Mal04}. 
It is worth noting that this inequality holds without requiring the centroid of the zeros of the polynomial to be at the origin.

In 1999, de Bruin and Sharma~\cite{dS99} conjectured a higher-order Schoenberg type inequality: If $\sum_{j=1}^{n}z_{j}=0$, then 
\begin{eqnarray}\label{bs}
\sum_{k=1}^{n-1}\left|w_{k}\right|^4 \le \frac{n-4}{n}\sum_{j=1}^{n}|z_{j}|^4+\frac{2}{n^2}\left(\sum_{j=1}^{n}|z_{j}|^2\right)^2.
\end{eqnarray}
Inequality~\eqref{bs} was first confirmed by   Cheung and Ng~\cite{CN06} using the D-companion matrix approach. An alternative proof via circulant matrices was given by  Kushel and Tyaglov~\cite{KT16}. In fact, Kushel and Tyaglov established a tighter inequality than~\eqref{bs}: 
\begin{eqnarray}\label{kt}
\sum_{k=1}^{n-1}\left|w_{k}\right|^4 \le \frac{n-4}{n}\sum_{j=1}^{n}|z_{j}|^4+\frac{1}{n^2}\left(\sum_{j=1}^{n}|z_{j}|^2\right)^2+\frac{1}{n^2}\left|\sum_{j=1}^{n}z_{j}^2\right|^2.
\end{eqnarray} In \cite[Section~5]{KT16}, Kushel and Tyaglov further posed an open problem concerning higher even orders. They wrote:
\begin{quote}
\emph{If $p(z)$ is the characteristic polynomial of a circulant matrix $C$, then by Schur's theorem, for the critical points $w_k$ of $p(z)$, one has 
\begin{equation}\label{problem1}
\sum_{k=1}^{n-1}|w_k|^{2m} \leqslant \operatorname{tr}\left(C_{n-1}^m (C_{n-1}^*)^m\right)\footnote{Here, \( C_{n-1} \) denotes the principal submatrix of order \( n - 1 \) of the circulant matrix \( C \), formed by taking its top-left \( (n - 1) \times (n - 1) \) block.}.
\end{equation}
However, even for $m=3$, it is a technically difficult problem to express $\tr\left(C_{n-1}^m (C_{n-1}^*)^m\right)$ in terms of the roots of the polynomial $p(z)$. ... Nevertheless, an estimate for $\sum_{k=1}^{n-1}|w_k|^6$ would be quite interesting.}
\end{quote}They also raised the question of whether Schoenberg type inequalities of odd orders could be established:
\begin{quote}
\emph{It is interesting to know whether it is possible to estimate
\begin{equation}\label{problem2}\displaystyle \sum_{k=1}^{n-1}|w_k|^{2m+1},\end{equation}
for $m=0,1,2,\ldots$.}
\end{quote}
We also note that in~\cite{LXZ21}, Lin, Xie, and Zhang presented a version of the inequality for all orders greater than or equal to $2$. However, their bound is not sharp even in the case of order $4$.

In this paper, we address both of the above questions for specific cases, namely \( m = 3 \) and \( m = 0 \), corresponding to orders \( 6 \) and \( 1 \), respectively.

The paper is organized as follows. In Section~\ref{secc2}, we provide an alternative proof of inequality~\eqref{kt}, and subsequently derive a sharp Schoenberg type inequality of order~$6$. This result resolves the case $m=3$ of the problem stated in~\eqref{problem1}. In Section~\ref{secc3}, we establish a Schoenberg type inequality of order~$1$, thereby addressing the case $m=0$ of the problem posed in~\eqref{problem2}. This is the first known Schoenberg type inequality of odd order, inspired by classical inequalities established in the late 1940s by de Bruijn and Springer~\cite{BS47}, and Erdős and Niven~\cite{EN48}. In Section~\ref{secc4}, we apply Schoenberg's inequality (more precisely, its equivalent form~\eqref{s}) to prove a special case of Sendov's conjecture. Final comments are provided in Section~\ref{secc5}.

\section{A Schoenberg Type Inequality of Order $6$}\label{secc2}

Let \( I \) denote the \( n \times n \) identity matrix, and let \( J \) be the \( n \times n \) matrix whose entries are all equal to one. The following observation, due to Komarova and Rivin~\cite[Lemma~5.7]{KR03}, plays an important role in our derivation.
\begin{lem}[\cite{KR03}]\label{lem1}
Let \( n \ge 2 \) and \( z_1, z_2, \ldots, z_n \in \mathbb{C} \). Let \( p(z) = \prod_{j=1}^n (z - z_j) \), and define \( D = \diag(z_1, z_2, \ldots, z_n) \). Then the characteristic polynomial of the matrix \( D\left(I - \frac{1}{n}J\right) \) has the same set of zeros as the polynomial \( zp'(z) \).
\end{lem}

Lemma~\ref{lem1} can be verified by direct computation. We refer the interested reader to \cite[Theorem 1.1]{CN06} for a similar matrix called D-companion matrix of a polynomial. A generalization of Lemma \ref{lem1} was given in \cite[Theorem 1.2]{CN10}, and has been applied to the study of the zeros of convex combinations of incomplete polynomials; see, for instance,~\cite{DE08, Zha24}.


Let \( A \) be an \( n \times n \) complex matrix. The eigenvalues and singular values of \( A \) are denoted by \( \lambda_j(A) \) and \( \sigma_j(A) \), respectively, for \( j = 1, \ldots, n \). We write \( \operatorname{tr}(A) \) for the trace of \( A \), and \( A^* \) for its conjugate transpose. The matrix \( A \) is said to be normal if \( A^*A = AA^* \). We also use the notation \( |A| := (A^*A)^{1/2} \) for the square root of \( A^*A \).

The following lemma is standard in matrix analysis.\begin{lem}[{\cite[p.~176]{HJ91}}] \label{lem2}  Let $A$ be any $n\times n$ complex matrix. Then for any $r>0$, $$\sum_{j=1}^n|\lambda_j(A)|^r\le \sum_{j=1}^n\sigma_j^r(A)=\tr \left(\left(A^*A\right)^{r/2}\right),$$ with equality  if and only if $A$ is normal \cite[p.~455]{HJ13}. The case $r=2$ of the inequality is called   Schur's inequality \cite[p.~102]{HJ13}. 
\end{lem}

With Lemma \ref{lem1} and Lemma \ref{lem2}, inequality~\eqref{kt} follows from a simple calculation. 

\medskip
\noindent
\textit{Proof of~\eqref{kt}.} Let $S=I-\frac{1}{n}J$ and let $A=(SDS)^2$, where $D =\diag(z_1, z_2, \ldots, z_n)$. Note that  $S$ is a projection. Then
\begin{eqnarray*}\sum_{k=1}^{n-1}\left|w_{k}\right|^4   &=&  \sum_{j=1}^n\left|\lambda_j(DS)\right|^4 \qquad \text{(by Lemma \ref{lem1})}  \\ &=&  \sum_{j=1}^n\left|\lambda_j(DS^2)\right|^4
  \\&=&   \sum_{j=1}^n\left|\lambda_j(SDS)\right|^4
    \\&=&  \sum_{j=1}^n|\lambda_j(A)|^2\le \tr \left(A^*A\right) \qquad \text{(by Lemma \ref{lem2})} \nonumber \\&=&   \tr\left(\left(I-\frac{1}{n}J\right)D^*\left(I-\frac{1}{n}J\right)D^*\left(I-\frac{1}{n}J\right)D\left(I-\frac{1}{n}J\right)D\right). 
\end{eqnarray*}
We need to compute the trace explicitly. Since $\sum_{j=1}^{n}z_{j}=0$, it is clear $JDJ=0$. Hence \begin{eqnarray*} \left(I-\frac{1}{n}J\right)D\left(I-\frac{1}{n}J\right)D=D^2-\frac{1}{n}DJD-\frac{1}{n}JD^2. \end{eqnarray*}
For any diagonal matrix \( E \), it holds that \( \operatorname{tr}(EJ) = \operatorname{tr}(E) \) and \( JEJ = (\operatorname{tr}(E))J \). Therefore,
\begin{eqnarray*}&& \tr\left(\left(I-\frac{1}{n}J\right)D^*\left(I-\frac{1}{n}J\right)D^*\left(I-\frac{1}{n}J\right)D\left(I-\frac{1}{n}J\right)D\right)
\\&=&\tr\left( \left(D^{*2}-\frac{1}{n}D^*JD^*-\frac{1}{n}JD^{*2}\right)\left(D^2-\frac{1}{n}DJD-\frac{1}{n}JD^2\right)\right)
\\&=&\tr \left(|D|^4-\frac{1}{n}D^*|D|^2JD-\frac{1}{n}D^{*2}JD^2\right.\\ &&\qquad-\frac{1}{n}D^*J|D|^2D+\frac{1}{n^2}D^*J|D|^2JD \\&&\qquad\left.-\frac{1}{n}J|D|^4+\frac{1}{n^2}JD^*|D|^2JD+\frac{1}{n^2}JD^{*2}JD^2\right)\\&=&\frac{n-4}{n}\sum_{j=1}^{n}|z_{j}|^4+\frac{1}{n^2}\left(\sum_{j=1}^{n}|z_{j}|^2\right)^2+\frac{1}{n^2}\left|\sum_{j=1}^{n}z_{j}^2\right|^2,  \end{eqnarray*} as desired. \qed

\begin{rem}In the preceding proof, if we instead take $A=\left(D\left(I-\frac{1}{n}J\right)\right)^2$, then we obtain only
\begin{eqnarray*} 
\sum_{k=1}^{n-1}\left|w_{k}\right|^4 \le \tr \left(A^*A\right)=\frac{n-3}{n}\sum_{j=1}^{n}|z_{j}|^4+\frac{1}{n^2}\left(\sum_{j=1}^{n}|z_{j}|^2\right)^2,
\end{eqnarray*}
which is even weaker than inequality~\eqref{bs}. This highlights the importance of choosing an appropriate matrix formulation.
\end{rem}

The main result of this section is the following Schoenberg type inequality of order $6$. 
\begin{thm}\label{st6}
Let \( z_1, z_2, \ldots, z_n \) be the zeros of \( p(z) \), and let \( w_1, w_2, \ldots, w_{n-1} \) be its critical points. Assume further that $\sum_{j=1}^n z_j=0$. Then 
\begin{align*}
\sum_{k=1}^{n-1} |w_k|^6 \leq\ 
&\frac{n - 6}{n} \sum_{j=1}^n |z_j|^6
+ \frac{6}{n^2} \left( \sum_{j=1}^n |z_j|^4 \right) \left( \sum_{j=1}^n |z_j|^2 \right) \\
&+ \frac{3}{n^2} \left| \sum_{j=1}^n z_j |z_j|^2 \right|^2
- \frac{2}{n^3} \left( \sum_{j=1}^n |z_j|^2 \right)^3, \tag{$\star$}
\end{align*}
with equality if and only if all \( z_j \) lie on a straight line in the complex plane.
\end{thm}

\begin{proof}
Let \( A = SDS \), where \( S = I - \frac{1}{n}J \) and \( D = \operatorname{diag}(z_1, z_2, \ldots, z_n) \). By Lemmas~\ref{lem1} and~\ref{lem2} with \( r = 6 \), we obtain
\[
\sum_{k=1}^{n-1}\left|w_{k}\right|^6 = \sum_{j=1}^n|\lambda_j(A)|^6 \le \operatorname{tr}\left((A^*A)^3\right) = \operatorname{tr}\left((SD^*SDS)^3\right)= \operatorname{tr}\left((SD^*SD)^3\right).
\]
Expanding \( \operatorname{tr}\left((S D^* S D)^3\right) \) yields the right-hand side of inequality~($\star$) as stated in the theorem. The detailed computation is provided in the Appendix.

By Lemma \ref{lem2}, equality holds if and only if $A$ is normal, i.e., 
\begin{eqnarray}\label{normal}
SD^*SDS=SDSD^*S.
\end{eqnarray} After simplification, (\ref{normal}) becomes $D^*JD=DJD^*$, that is $\bar{z}_iz_j=\bar{z}_jz_i$ for all $i\ne j$, which means all zeros of $p(z)$ lie on a straight line.  \end{proof}

 \begin{rem}
 In fact, one can obtain a tighter, though more complicated, inequality than that in Theorem~\ref{st6}, as described below. In the proof of Theorem~\ref{st6}, if we instead apply Schur's inequality directly to $A^3$, we obtain the bound \(\sum_{k=1}^{n-1}\left|w_{k}\right|^6 = \sum_{j=1}^n|\lambda_j(A^3)|^2 \le \operatorname{tr}\left( (A^*)^3 A^3 \right).\) This inequality is indeed tighter, as it follows from a well-known consequence of Horn's inequalities that \(\operatorname{tr}\left( (A^*)^3 A^3 \right) \le \operatorname{tr}\left( (A^*A)^3 \right)\); see, for example, \cite[p.~177, p.~190]{HJ91}. Let \( \Re z \) denote the real part of a complex number. Then the upper bound \( \operatorname{tr}\left( (A^*)^3 A^3 \right) \) can be computed explicitly. \begin{align*}
\operatorname{tr}\left( (A^*)^3 A^3 \right) = \ & \frac{n-6}{n}   \sum_{j=1}^{n} |z_j|^6   + \frac{2}{n^2} \left( \sum_{j=1}^{n} |z_j|^4 \right) \left( \sum_{j=1}^{n} |z_j|^2 \right)\\& +\frac{4}{n^2} \Re \left\{ \left( \sum_{j=1}^{n} z_j^2 |z_j|^2 \right) \left( \sum_{j=1}^{n} \bar{z}_j^2 \right) \right\}  + \frac{2}{n^2} \left| \sum_{j=1}^{n} z_j |z_j|^2 \right|^2 \\
& + \frac{1}{n^2} \left| \sum_{j=1}^{n} z_j^3 \right|^2 - \frac{2}{n^3} \left( \sum_{j=1}^{n} |z_j|^2 \right) \left| \sum_{j=1}^{n} z_j^2 \right|^2. \tag{$\star\star$}
\end{align*}The detailed computation leading to \((\star\star)\) is provided in the Appendix. The heavy computation and the complicated formula of the upper bound obviously stop us from finding higher order Schoenberg type inequalities using this approach. 
\end{rem}

\section{A Schoenberg Type Inequality of Order $1$}\label{secc3}

Let \( \bold{v} = (v_1, \ldots, v_m) \), where \( v_1, \ldots, v_m \in \mathbb{C} \). We denote by \( |\bold{v}| = (|v_1|, \ldots, |v_m|) \) the entrywise modulus of \( \bold{v} \), and by \( e_k(\bold{v}) \) the \( k \)th elementary symmetric function of \( v_1, \ldots, v_m \) (see, e.g.,~\cite[p.~54]{HJ13}).

Let \( p(z) = \prod_{k=1}^n (z - z_k) = z^n + \sum_{k=1}^n (-1)^k e_k(\bold{z}) z^{n-k} \). Then its derivative is
\[
p'(z) = nz^{n-1} + \sum_{k=1}^{n-1} (n-k)(-1)^k e_k(\bold{z}) z^{n-1-k}.
\]
On the other hand, since \( p'(z) = n \prod_{k=1}^{n-1} (z - w_k) = nz^{n-1} + n \sum_{k=1}^{n-1} (-1)^k e_k(\bold{w}) z^{n-1-k} \), comparing coefficients yields
\[
e_k(\bold{w}) = \frac{n-k}{n} e_k(\bold{z}), \qquad k = 1, \ldots, n-1.
\] In particular, the case \( k = 1 \) gives the well-known identity:
\[
\frac{1}{n-1} \sum_{k=1}^{n-1} w_k = \frac{1}{n} \sum_{j=1}^n z_j,
\]
i.e., the arithmetic mean of the critical points equals that of the zeros.

de Bruijn and Springer~\cite[Theorem~4]{BS47}, and independently Erdős and Niven~\cite{EN48}, showed that
\begin{equation}\label{bsen}
\sum_{k=1}^{n-1} |w_k| \le \frac{n-1}{n} \sum_{j=1}^n |z_j|.
\end{equation}
In our notation, this is
\begin{equation}\label{extextext}
e_1(|\bold{w}|) \le \frac{n-1}{n} e_1(|\bold{z}|).
\end{equation} 

The following result provides an extension of inequality~\eqref{extextext} to all elementary symmetric functions:

\begin{thm}\label{ek}
Let \( z_1, z_2, \ldots, z_n \) be the zeros of \( p(z) \), and let \( w_1, w_2, \ldots, w_{n-1} \) be its critical points. Then 
 \begin{eqnarray*}
   e_k(|\bold{w}|)\le\frac{n-k}{n} e_k(|\bold{z}|), \qquad  k=1, \ldots, n-1.
   \end{eqnarray*}  
\end{thm}

Indeed, Malamud~\cite[Theorem~4.10]{Mal04} proved a result more general than Theorem~\ref{ek}. Here, we provide an alternative proof of Theorem~\ref{ek}. Our approach is based on a weak log-majorization result, originally proved by Cheung and Ng~\cite[Theorem~2.1]{CN06} under the assumption that \( p(0) = 0 \), and later extended by Pereira~\cite[Theorem~2.1]{Per07} to arbitrary polynomials. The version we use is stated as Theorem~\ref{t1} below. A proof via the circulant matrix method can be found in~\cite{LXZ21}. We refer the reader to~\cite{MOA11} for basic theory of majorization.

 \begin{thm}[\cite{Per07}]\label{t1}  Let $p(z)$ be a polynomial of degree $n\ge 2$ with zeroes $z_j$, $j = 1, \ldots, n$, and
  critical points $w_j$, $j = 1, \ldots, n-1$, listed in descending order of modulus. If $q(z)$ is a polynomial with zeroes $|z_j|$, $j = 1, \ldots, n$,  and critical points
  $\xi_j$, $j = 1, \ldots, n-1$,  listed in descending order of modulus,  then 
   	\begin{eqnarray*} \prod_{j=1}^{k}|w_j|\le \prod_{j=1}^{k}\xi_j, \quad k=1,\ldots,n-1.  	\end{eqnarray*}
  \end{thm}

The following consequence of Theorem~\ref{t1} will be used in our argument.

 \begin{cor}\label{cc1}  With the same condition as in Theorem \ref{t1}, we have \begin{eqnarray*} e_k(|\bold{w}|)\le e_k(\boldsymbol{\xi}), \qquad k=1, \ldots, n-1.	\end{eqnarray*}
  \end{cor}
\begin{proof}Without loss of generality, we assume $|w_{n-1}|>0$. The general case follows by a continuity argument. We may choose an entrywise positive vector $\bold{x}=(x_1, \ldots, x_{n-2}, x_{n-1})$, where $x_{j}=\xi_j$ for $j=1, \ldots, n-2$, such that $\prod_{j=1}^{n-1}|w_j|= \prod_{j=1}^{n-1}x_j$.  Then  $(\log |w_1|, \ldots, \log |w_{n-1}|)$ is majorized by  $(\log x_1, \ldots, \log x_{n-1})$. By \cite[Theorem 9]{Rov12}, $e_{k}(e^{t_1}, \ldots, e^{t_{n-1}})$ is Schur convex over $t_i\in  \mathbb{R}$, we have  $e_k(|\bold{w}|)\le e_k(\bold{x})$. Noting that $e_k:\mathbb{R}^{n-1}_+\to \mathbb{R}$ is monotone increasing  with respect to the entries, we get $e_k(\bold{x})\le e_k(\boldsymbol{\xi})$, since clearly we have $x_{n-1}\le \xi_{n-1}$.  \end{proof}

We are now ready to prove Theorem~\ref{ek}.

\begin{proof}[Proof of Theorem~\ref{ek}]
The result follows immediately from Corollary~\ref{cc1} and the identity
\[
e_k(\boldsymbol{\xi}) = \frac{n - k}{n} e_k(|\boldsymbol{z}|), \qquad k = 1, \ldots, n - 1.
\]
\end{proof}

The highlighted result of this section is the following Schoenberg type inequality of order~$1$, which is stronger than~\eqref{bsen} under an additional condition on the centroid of the zeros.

\begin{thm} 
Let \( z_1, z_2, \ldots, z_n \) be the zeros of \( p(z) \), and let \( w_1, w_2, \ldots, w_{n-1} \) be its critical points. Assume further that $\sum_{j=1}^n z_j=0$, then 
 \begin{eqnarray}\label{st1}
\sum_{k=1}^{n-1}\left|w_k\right| \leq   \sqrt{\frac{n-2}{n}} \sum_{j=1}^n\left|z_j\right|.\end{eqnarray}
 \end{thm}
\begin{proof} It suffices to prove that
 \begin{eqnarray*}
\left(\sum_{k=1}^{n-1}\left|w_k\right|\right)^2 \leq    \frac{n-2}{n}  \left(\sum_{j=1}^n\left|z_j\right|\right)^2, \end{eqnarray*} which expands as
 \begin{eqnarray*}
\sum_{k=1}^{n-1}\left|w_k\right|^2+2  \sum_{1\leq i<j\leq n-1}\left|w_i w_j\right| \leq \frac{n-2}{n} \sum_{k=1}^n\left|z_k\right|^2+ \frac{2(n-2)}{n}\sum_{1\leq i<j\leq n}\left|z_i z_j\right|. \end{eqnarray*} 
Note that the case \( k = 2 \) of Theorem~\ref{ek} gives
$$\sum_{1\leq i<j\leq n-1}\left|w_i w_j\right| \leqslant \frac{n-2}{n} \sum_{1\leq i<j\leq n}\left|z_i z_j\right|.$$ Taking into account this and  Schoenberg's inequality (\ref{s0}), the desired result follows. 
\end{proof}
Inequality~\eqref{st1} is sharp. For example, take \( p(z) = z^{n - 2}(z^2 - 1) \) with \( n \ge 2 \). However, the condition that all zeros \( z_j \) lie on a straight line through the origin is not sufficient for equality to hold. For instance, when \( p(z) = (z^2 - 1)^2 \), the two sides of inequality~\eqref{st1} are not equal.
  
\begin{rem}
Pereira~\cite[Corollary~5.5]{Per03}, solving a conjecture of de Bruijn and Springer~\cite{BS47}, showed that the absolute value function in inequality~\eqref{bsen} can be replaced by a much more general convex function \( \Phi: \mathbb{C} \to \mathbb{R} \). It is natural to ask whether a similar generalization holds for inequality~\eqref{st1}.
\end{rem}

\section{A Special Case of Sendov’s Conjecture}\label{secc4}

Sendov's conjecture is one of the most longstanding open problems in the geometry of polynomials; see~\cite{Tao22} and the references therein. Since a rotation of the complex plane preserves the relative configuration of the zeros and critical points of a polynomial \( p(z) \), it suffices to consider the following reformulation of the conjecture.

\noindent{\bf Sendov's Conjecture.}
Let \( p(z) = (z - a) \prod_{j = 1}^{n - 1} (z - z_j) \) with \( 0 \le a \le 1 \), \( |z_j| \le 1 \) for \( j = 1, \ldots, n - 1 \), and \( n \ge 2 \). Then the closed disk \( |z - a| \le 1 \) contains a critical point of \( p(z) \).

In this section, we make use of Schoenberg's inequality  (\ref{s}) to prove a special case of Sendov's conjecture.

\begin{thm}\label{specialsc}Let \( p(z) = (z - a) \prod_{j = 1}^{n - 1} (z - z_j) \) with \( 0 \le a \le 1 \), \( |z_j| \le 1 \) for \( j = 1, \ldots, n - 1 \), and \( n \ge 2 \). If further, $\Re\sum_{j=1}^{n-1}z_j\ge \frac{n-2}{2}a$, then the open disk 
$|z-a| < 1$ contains a critical point of $p(z)$. 
\end{thm}
\begin{proof}Let \( w_1, \ldots, w_{n - 1} \) be the critical points of \( p(z) \), and define \( q(z) := p(z + a) \). Applying inequality~\eqref{s} to \( q(z) \) and then using the Cauchy--Schwarz inequality, we obtain
\begin{eqnarray}\label{scs}
 \sum_{k=1}^{n-1}|w_k-a|^2&\le&   \frac{n-2}{n}\sum_{j=1}^{n-1}|z_{j}-a|^2+\frac{1}{n^2}\left|\sum_{j=1}^{n-1}(z_{j}-a)\right|^2 \nonumber \\&\le &\frac{n^2 - n -1}{n^2}\sum_{j=1}^{n-1}|z_{j}-a|^2.
\end{eqnarray}
If $w_k=a$ for some $k$, there is nothing to prove; we assume otherwise. Now 
\begin{eqnarray*}
 &&\sum_{k=1}^{n-1}|w_k-a|^{-2}\left(a^2+\sum_{j=1}^{n-1}|z_{j}|^2\right)\\&=&    
 \sum_{k=1}^{n-1}|w_k-a|^{-2}\left(a^2+\sum_{j=1}^{n-1}|z_{j}-a+a|^2\right)
 \\&=&    
  \sum_{k=1}^{n-1}|w_k-a|^{-2}\left(na^2+2a\Re\sum_{j=1}^{n-1}(z_{j}-a)+\sum_{j=1}^{n-1}|z_{j}-a|^2\right)
 \\&\ge&  \sum_{k=1}^{n-1}|w_k-a|^{-2}\left(na^2+2a\Re\sum_{j=1}^{n-1}(z_{j}-a)+\frac{n^2}{n^2 - n -1}\sum_{k=1}^{n-1}|w_{k}-a|^2\right) \\&\ge&  \sum_{k=1}^{n-1}|w_k-a|^{-2}\left( \frac{n^2}{n^2 - n -1}\sum_{k=1}^{n-1}|w_{k}-a|^2\right)\\&\ge& 
 \frac{n^2(n-1)^2}{n^2 - n -1} >n(n-1) ,
 \end{eqnarray*}where the first inequality follows from~\eqref{scs}, the second from the given condition, and the third from the Cauchy-Schwarz inequality. Therefore, \begin{eqnarray*}
   \sum_{k=1}^{n-1}|w_k-a|^{-2}\left(a^2+\sum_{j=1}^{n-1}|z_{j}|^2\right) > n(n-1). \end{eqnarray*} Since \( a^2 + \sum_{j = 1}^{n - 1} |z_j|^2 \le n \), it follows that
\begin{equation}\label{c1}
\sum_{k = 1}^{n - 1} |w_k - a|^{-2} > n - 1.
\end{equation}
In particular, there exists a \( k \) such that \( |w_k - a| < 1 \), as desired. 
\end{proof}

We end this section with a few remarks.


\begin{rem}\label{remalternativeproof}
An alternative and simpler proof of Theorem~\ref{specialsc} proceeds from inequality~\eqref{scs}. Observe that
\begin{eqnarray*} 
 \sum_{k=1}^{n-1}|w_k-a|^2&\le&    \frac{n^2 - n -1}{n^2}\sum_{j=1}^{n-1}|z_{j}-a|^2\\&= & \frac{n^2 - n -1}{n^2}\left(\sum_{j=1}^{n-1}|z_j|^2-2a\Re \sum_{j=1}^{n-1}z_j+(n-1)a^2\right) \\&\le & \frac{n^2 - n -1}{n^2}\left(n-2a\Re \sum_{j=1}^{n-1}z_j+(n-2)a^2\right) < n-1,
\end{eqnarray*}
where the final inequality follows from the assumption \( \Re \sum_{j=1}^{n-1} z_j \ge \frac{n - 2}{2} a \). Therefore,
\begin{equation}\label{c2}
\sum_{k = 1}^{n - 1} |w_k - a|^2 < n - 1,
\end{equation}
which implies that there exists a \( k \) such that \( |w_k - a| < 1 \), as desired. \qed
\end{rem}

The two different proofs of Theorem~\ref{specialsc} presented above were motivated by the following observation. Let \( x_1, \ldots, x_m \) be positive real numbers. Their power mean with exponent \( p \) is defined as \(M_p(x_1, \ldots, x_m) = \left( \frac{1}{m} \sum_{i=1}^m x_i^p \right)^{1/p}\). We denote \( M_p := M_p(|w_1 - a|, \ldots, |w_{n-1} - a|) \) for brevity. Sendov's conjecture is equivalent to proving that
\[
\min_{1 \le k \le n-1} |w_k - a| \le 1,
\]
which is the same as
\[
M_{-\infty} \le 1.
\]
By the monotonicity of power means (i.e., \( M_p \le M_q \) for \( p < q \)), we obtain
\[
M_{-\infty} \le M_{-2} \le M_{2},
\]
where \( M_{-2} \) and \( M_2 \) correspond to inequalities~\eqref{c1} and~\eqref{c2}, respectively.

Without the additional assumption \( \Re \sum_{j = 1}^{n - 1} z_j \ge \frac{n - 2}{2}a \), the inequality \( M_2 \le 1 \) does not hold in general. For instance, consider the polynomial \( p(z) = (z - a)(z + 1)^{n - 1} \) with \( a \) close to \( 1 \), where the bound fails.

On the other hand, it remains unclear whether the inequality \( M_{-2} \le 1 \) always holds. Despite considerable numerical testing, we have not been able to identify a counterexample. It is possible that this inequality represents a stronger conjecture than Sendov's original conjecture. This observation also suggests that, if possible, establishing Schoenberg type inequalities of negative order, particularly of order \( -2 \), may offer interesting insight into the resolution of Sendov's conjecture.

\section{Comments}\label{secc5}
In~\cite[Theorem~3.2]{LXZ21}, the authors established an inequality for general powers \( r \ge 2 \) and explicitly posed the question of whether further improvements are possible.
\begin{thm}[\cite{LXZ21}]\label{psum} Let \( z_1, z_2, \ldots, z_n \) be the zeros of \( p(z) \), and let \( w_1, w_2, \ldots, w_{n-1} \) be its critical points. Then for $r\ge 2$, \begin{eqnarray}\label{lxz}
  \sum_{k=1}^{n-1}\left|w_{k}\right|^{r} \le \dfrac{(n-1)^{r-2}}{n^r}\left|\sum_{j=1}^{n}z_{j}\right|^r+\dfrac{(n-1)^{r-2}(n-2)}{n^{r/2}}\left( \sum_{j=1}^{n}|z_{j}|^2\right) ^{r/2}.
  \end{eqnarray} \end{thm}
We now present the following improvement, which provides a tighter bound under the same conditions.
\begin{prop}
Under the same conditions as in Theorem~\ref{psum}, the following inequality holds for all \( r \ge 2 \):
\begin{equation}\label{impro}
\sum_{k = 1}^{n - 1} |w_k|^r \le \left( \frac{n - 2}{n} \sum_{j = 1}^n |z_j|^2 + \frac{1}{n^2} \left| \sum_{j = 1}^n z_j \right|^2 \right)^{r/2}.
\end{equation}
\end{prop}

 \begin{proof}
 By Jensen's inequality and (\ref{s}), we  have 
    \begin{eqnarray*}\left(\sum_{k=1}^{n-1}\left|w_{k}\right|^r\right)^{1/r} \le \left( \sum_{k=1}^{n-1}\left|w_{k}\right|^2\right)^{1/2} \le \left(\frac{n-2}{n}\sum_{j=1}^{n}|z_{j}|^2+\frac{1}{n^2}\left|\sum_{j=1}^{n}z_{j}\right|^2\right)^{1/2},
     \end{eqnarray*}  thus  (\ref{impro}) follows.
 \end{proof}
  When $n=2$, (\ref{impro}) and (\ref{lxz}) coincide.   To show  (\ref{impro}) is tighter than (\ref{lxz}) for $n\ge 3$, we note, by the monotonicity of power means, that \begin{eqnarray*} \left(\frac{n-2}{n}\sum_{j=1}^{n}|z_{j}|^2+\frac{1}{n^2}\left|\sum_{j=1}^{n}z_{j}\right|^2\right)^{r/2}\le 2^{r/2-1}\left(\frac{(n-2)^{r/2}}{n^{r/2}}\left( \sum_{j=1}^{n}|z_{j}|^2\right) ^{r/2}+\frac{1}{n^r}\left|\sum_{j=1}^{n}z_{j}\right|^r\right).  \end{eqnarray*} It remains to verify ${\displaystyle 2^{r/2-1}\frac{(n-2)^{r/2}}{n^{r/2}}\le \dfrac{(n-1)^{r-2}(n-2)}{n^{r/2}}}$ and ${\displaystyle\frac{2^{r/2-1}}{n^r}\le \dfrac{(n-1)^{r-2}}{n^r}}$, which is a simple exercise.

\section*{Acknowledgments}

I would like to express my sincere gratitude to Prof.~Minghua Lin for numerous helpful discussions and valuable suggestions. I also thank Yuxiang Huang for insightful comments on an earlier version of this paper.

\section*{Appendix} 

\noindent {\it Proof of ($\star$).} We now compute \( \operatorname{tr}\left((S D^* S D)^3\right) \) and show that it equals the right-hand side of inequality~($\star$). First, we compute:\begin{eqnarray*}
S D^* S D
&=& \left( I - \frac{1}{n}J \right) D^* \left( I - \frac{1}{n}J \right) D \\
&=& \left( D^* - \frac{1}{n} JD^*  \right) \left( D - \frac{1}{n} JD \right) \\
&=& |D|^2 - \frac{1}{n} D^* JD - \frac{1}{n} J|D|^2.
\end{eqnarray*}
Consequently,
\begin{eqnarray*}
(S D^* S D)^2 
&=& \left( |D|^2 - \frac{1}{n} D^* JD - \frac{1}{n} J|D|^2 \right)^2 \\
&=& |D|^4 - \frac{1}{n} |D|^2 D^* JD - \frac{1}{n} |D|^2J|D|^2 - \frac{1}{n} D^* JD |D|^2 \\
&&+ \frac{1}{n^2} D^* J |D|^2 J D - \frac{1}{n} J|D|^4 + \frac{1}{n^2} J |D|^2 D^* JD + \frac{1}{n^2} J |D|^2 J|D|^2.
\end{eqnarray*}
Therefore,
\begin{eqnarray*}
\tr\left( (S D^* S D)^3 \right) 
&=& \tr\left(\left( |D|^2 - \frac{1}{n} D^* JD - \frac{1}{n} J |D|^2 \right) \times \right.\\
&& \left( |D|^4 - \frac{1}{n} |D|^2 D^* JD - \frac{1}{n} |D|^2 J |D|^2 - \frac{1}{n} D^* JD |D|^2 + \frac{1}{n^2} D^* J |D|^2 J D \right. \\
&& \left.\left. - \frac{1}{n} J |D|^4 + \frac{1}{n^2} J |D|^2 D^* JD + \frac{1}{n^2} J |D|^2 J |D|^2 \right)\right) \\
&=& \tr\left(|D|^6 - \frac{1}{n} |D|^4 D^* JD  - \frac{1}{n} |D|^4 J |D|^2  - \frac{1}{n} |D|^2 D^* JD |D|^2  + \frac{1}{n^2} |D|^2 D^* J |D|^2 J D \right.\\
&&  \left.- \frac{1}{n} |D|^2 J |D|^4 + \frac{1}{n^2} |D|^2 J |D|^2 D^* JD + \frac{1}{n^2}|D|^2J|D|^2J|D|^2 - \frac{1}{n} D^* JD |D|^4 \right.\\
&&  \left.+ \frac{1}{n^2} D^* J|D|^4 JD + \frac{1}{n^2} D^* JD |D|^2 J |D|^2  + \frac{1}{n^2} D^* J |D|^2 J D |D|^2 \right.\\
&&  \left.- \frac{1}{n^3} D^* J |D|^2 J |D|^2JD   - \frac{1}{n} J |D|^6   + \frac{1}{n^2} J |D|^4 J |D|^2   + \frac{1}{n^2} J |D|^2 D^* JD |D|^2  \right.\\
&&  \left.+ \frac{1}{n^2} J |D|^2 J |D|^4  - \frac{1}{n^3} J |D|^2 J |D|^2 J |D|^2 \right)\\&=& \tr \left( |D|^6 \right) - \frac{1}{n} \tr \left( |D|^6 \right) - \frac{1}{n} \tr \left( |D|^6 \right) - \frac{1}{n} \tr \left( |D|^6 \right) \\
&& + \frac{1}{n^2} \tr \left( |D|^2 \right) \tr \left( |D|^4 \right) - \frac{1}{n} \tr \left( |D|^6 \right) + \frac{1}{n^2} \tr \left( |D|^2 D^* \right) \tr \left( |D|^2 D \right) \\
&& + \frac{1}{n^2} \tr \left( |D|^2 \right) \tr \left( |D|^4 \right) - \frac{1}{n} \tr \left( |D|^6 \right) + \frac{1}{n^2} \tr \left( |D|^2 \right) \tr \left( |D|^4 \right) \\
&& + \frac{1}{n^2} \tr \left( |D|^2 D \right) \tr \left( |D|^2 D^* \right) + \frac{1}{n^2} \tr \left( |D|^2 \right) \tr \left( |D|^4 \right) - \frac{1}{n^3} \left(\tr  |D|^2 \right)^3 \\
&& - \frac{1}{n} \tr \left( |D|^6 \right) + \frac{1}{n^2} \tr \left( |D|^4 \right) \tr \left( |D|^2 \right) + \frac{1}{n^2} \tr \left( |D|^2 D^* \right) \tr \left( |D|^2 D \right) \\
&& + \frac{1}{n^2} \tr \left( |D|^4 \right) \tr \left( |D|^2 \right) - \frac{1}{n^3} \left(\tr  |D|^2 \right)^3\\
&=&\frac{n-6}{n}\left(\sum_{j=1}^n\left|z_j\right|^6\right)+\frac{6}{n^2}\left(\sum_{j=1}^n\left|z_j\right|^4\right)\left(\sum_{j=1}^n\left|z_j\right|^2\right)\\&& +\frac{3}{n^2}\left|\sum_{j=1}^n z_j|z_j|^2\right|^2-\frac{2}{n^3}\left(\sum_{j=1}^n\left|z_j\right|^2\right)^3.
\end{eqnarray*}\qed

\noindent {\it Proof of  ($\star\star$).} We now compute \( \operatorname{tr}\left( (A^*)^3 A^3 \right) \), as claimed in inequality~($\star\star$). First, we compute:
\begin{eqnarray*} \tr\left( \left(A^*\right)^3A^3\right)&=&\tr\left( (SD^*S)^3(SDS)^3\right)   \\&=&   \tr\left( \left[\left(I-\frac{1}{n}J\right)D^*\left(I-\frac{1}{n}J\right)D^*\left(I-\frac{1}{n}J\right)D^*\right]\right.\\&&\left.\qquad \times \left[\left(I-\frac{1}{n}J\right)D\left(I-\frac{1}{n}J\right)D\left(I-\frac{1}{n}J\right)D\right]\right). \end{eqnarray*}
Now,  \begin{eqnarray*} &&\left(I-\frac{1}{n}J\right)D\left(I-\frac{1}{n}J\right)D\left(I-\frac{1}{n}J\right)D\\&=& \left(D^2-\frac{1}{n}DJD-\frac{1}{n}JD^2\right)\left(D-\frac{1}{n}JD\right)\\&=&  D^3-\frac{1}{n}DJD^2-\frac{1}{n}JD^3-\frac{1}{n}D^2JD+\frac{1}{n^2}JD^2JD. \end{eqnarray*}
Therefore, \begin{eqnarray*} &&\tr\left(\left[\left(I-\frac{1}{n}J\right)D^*\left(I-\frac{1}{n}J\right)D^*\left(I-\frac{1}{n}J\right)D^*\right]\right.\\&&\left.\qquad \times \left[\left(I-\frac{1}{n}J\right)D\left(I-\frac{1}{n}J\right)D\left(I-\frac{1}{n}J\right)D\right]\right)\\&=&\tr \left(\left(D^{*3}-\frac{1}{n}D^*JD^{*2}-\frac{1}{n}JD^{*3}-\frac{1}{n}D^{*2}JD^*+\frac{1}{n^2}JD^{*2}JD^*\right)\right. \\&&\left.\qquad \times \left(D^3-\frac{1}{n}DJD^2-\frac{1}{n}JD^3-\frac{1}{n}D^2JD+\frac{1}{n^2}JD^2JD\right)\right) \\&=& \tr\left(|D|^6 - \frac{1}{n} D^{*2} |D|^2 JD^2 - \frac{1}{n}D^{*3} J D^3 - \frac{1}{n} D^*|D|^4 JD + \frac{1}{n^2} D^{*3} J D^2 JD \right. \\
&&\left. - \frac{1}{n}D^* J|D|^4 D+ \frac{1}{n^2} D^*JD^* |D|^2 JD^2 + \frac{1}{n^2} D^* JD^{*2} J D^3 + \frac{1}{n^2} D^* J |D|^4 JD-\frac{1}{n^3}D^* J D^{*2} J D^2 JD\right. \\
&&\left. - \frac{1}{n} J |D|^6  + \frac{1}{n^2} JD^{*2}|D|^2 JD^2 + \frac{1}{n^2} J D^{*3} JD^3 \right. \\
&&\left. - \frac{1}{n}D^{*2} J|D|^2 D^2 + \frac{1}{n^2} D^{*2} J|D|^2  JD^2 + \frac{1}{n^2}D^{*2} J |D|^2 D JD\right. \\
&& \left.+ \frac{1}{n^2} J D^{*2} J|D|^2 D^2 - \frac{1}{n^3} JD^{*2} J |D|^2 J D^2 \right)\\&=& \tr(|D|^6) - \frac{1}{n} \tr(|D|^6) - \frac{1}{n} \tr(|D|^6) - \frac{1}{n} \tr(|D|^6)  + \frac{1}{n^2} \tr(D^2) \tr(|D|^2 D^{*2}) \\
&& - \frac{1}{n} \tr(|D|^6) + \frac{1}{n^2} |\tr(|D|^2 D)|^2 + \frac{1}{n^2} \tr(D^{*2}) \tr(|D|^2 D^2)+ \frac{1}{n^2} \tr(|D|^4) \tr(|D|^2)- \frac{1}{n^3} |\tr(D^2)|^2 \tr(|D|^2) \\
&&  - \frac{1}{n} \tr(|D|^6) + \frac{1}{n^2} \tr(D^{*2}|D|^2) \tr(D^2) + \frac{1}{n^2} |\tr(D^3)|^2 \\
&& - \frac{1}{n} \tr(|D|^6)+ \frac{1}{n^2} \tr(|D|^2) \tr(|D|^4)+ \frac{1}{n^2} |\tr(|D|^2 D)|^2\\
&& + \frac{1}{n^2} \tr(|D|^2 D^2) \tr(D^{*2})- \frac{1}{n^3} \tr(|D|^2) |\tr(D^2)|^2\\&=& \frac{n-6}{n} \left( \sum_{j=1}^{n} |z_j|^6 \right) + \frac{2}{n^2} \left( \sum_{j=1}^{n} |z_j|^4 \right) \left( \sum_{j=1}^{n} |z_j|^2 \right)+ \frac{4}{n^2} \Re \left\{ \left( \sum_{j=1}^{n} z_j^2 |z_j|^2 \right) \left( \sum_{j=1}^{n} \bar{z_j}^2 \right) \right\} \\
&& + \frac{2}{n^2} \left| \sum_{j=1}^{n} z_j |z_j|^2 \right|^2 + \frac{1}{n^2} \left| \sum_{j=1}^{n} z_j^3 \right|^2 - \frac{2}{n^3} \left( \sum_{j=1}^{n} |z_j|^2 \right) \left| \sum_{j=1}^{n} z_j^2 \right|^2.
\end{eqnarray*}\qed

\end{document}